\documentclass[11pt,reqno]{amsart}

\usepackage{amsmath,amssymb,amsfonts,amsthm,mathtools,latexsym,enumerate,lmodern,hyperref,xcolor}

\usepackage[utf8]{inputenc}
\usepackage[T1]{fontenc}

\newcommand\A{\mathrm{A}}
\newcommand\bbF{\mathbb{F}}
\newcommand\C{\mathrm{C}}
\newcommand\D{\mathrm{D}}

\newcommand\M{\mathrm{M}}
\newcommand\Nor{\mathbf{N}}
\newcommand\PSL{\mathrm{PSL}}
\newcommand\Q{\mathrm{Q}}
\newcommand\SL{\mathrm{SL}} \newcommand\Sy{\mathrm{S}} \newcommand\Sym{\mathrm{Sym}}

\newtheorem{theorem}{Theorem}[section]
\newtheorem{lemma}[theorem]{Lemma}
\newtheorem{proposition}[theorem]{Proposition}

\newtheorem{problem}[theorem]{Problem}
\newtheorem{conjecture}[theorem]{Conjecture}
\theoremstyle{definition}
\newtheorem{question}[theorem]{Question}

\newtheorem{construction}[theorem]{Construction}
\newtheorem{example}[theorem]{Example}
\newtheorem*{remark}{Remark}

\begin{document}

\title[Elusive groups from non-split extensions]{Elusive groups from non-split extensions}

\author{Jiyong Chen}
\address{(Chen) School of Mathematical Sciences\\Xiamen University\\Xiamen\\P. R. China }
\email{chenjy1988@xmu.edu.cn}

\author{Melissa Lee}
\address{(Lee) School of Mathematics\\Monash University\\Clayton VIC 3800\\Australia}
\email{melissa.lee@monash.edu}

\author{\DJ or\dj e Mitrovi\' c}
\address{(Mitrovi\' c) Department of Mathematics\\University of Auckland\\Auckland\\New Zealand }
\email{dmit755@aucklanduni.ac.nz}

\author[O'Brien]{E.~A.~O'Brien}
\address{(O'Brien) Department of Mathematics\\University of Auckland\\Auckland\\New Zealand}
\email{e.obrien@auckland.ac.nz}

\author[Xia]{Binzhou Xia}
\address{(Xia) School of Mathematics and Statistics\\The University of Melbourne\\Parkville, VIC 3010\\Australia}
\email{binzhoux@unimelb.edu.au}

\begin{abstract}
A finite transitive permutation group is elusive if it contains no derangements of prime order. These groups are closely related to a longstanding open problem in algebraic graph theory known as the Polycirculant Conjecture, which asserts that no elusive group is $2$-closed. Existing constructions of elusive groups mostly arise from split extensions. In this paper, we initiate the construction of elusive groups via non-split extensions. As a demonstration, we construct elusive groups of new degrees, namely $p^{3k-4}(p+1)/2$ for each Mersenne prime $p\geq7$ and integer $k\geq2$. We also construct the first examples of elusive groups with odd degree, namely $3^{k+1}\cdot5^2$, and twice odd degree, namely $2\cdot3^{k + 1}\cdot5^2$ for each $k\geq1$. We conclude by proposing further problems to advance this new direction of research.


\end{abstract}

\maketitle

\section{Introduction}

Throughout this paper, all permutation groups act on finite sets.
A classical theorem of Jordan~\cite{Jordan1872} from 1872 states that every transitive permutation group on at least two points has derangements---elements with no fixed points. Jordan’s theorem has found applications in a wide range of mathematical areas (see, for instance, Serre~\cite{Serre2003}), while it admits a simple proof by the elementary orbit counting lemma.

In sharp contrast, a deep strengthening of Jordan’s theorem by Fein, Kantor and Schacher~\cite{FKS1981} in 1981 relies on the Classification of Finite Simple Groups to assert the existence of derangements of prime power order for every transitive permutation group. The condition of prime power order cannot be improved to prime order, since there exist transitive permutation groups with no derangements of prime order.
For example, the only derangements of the Mathieu group $\M_{11}$ in its primitive action on $12$ points are of order $4$ or $8$. An \emph{elusive group} is a transitive permutation group with no derangements of prime order. These groups have attracted significant attention in the literature~\cite{AG2021,DM2011,Giudici2003,Giudici2007,GK2009,GMPV2015,Klin1998,MS1998,Xu2009}.

Recall that the \emph{$2$-closure} of a permutation group $G$ on a set $\Omega$ is the largest subgroup of $\Sym(\Omega)$ that has the same orbits on $\Omega \times \Omega$ as $G$. The group $G$ is \emph{$2$-closed} if $G$ equals its $2$-closure. A major motivation for studying elusive groups arises from the following conjecture of Klin~\cite[Problem~BCC15.12]{Cameron1997} (see also~\cite{CGJKKMN2002}), which extends the earlier Semiregularity Problem posed in~\cite{Marusic1981} and~\cite{Jordan1988}.

\begin{conjecture}[Polycirculant Conjecture]
There are no $2$-closed elusive groups.
\end{conjecture}

The Polycirculant Conjecture remains open, although various partial results have been obtained; we refer to~\cite{AAS2019} for a detailed survey. Notably, Giudici~\cite{Giudici2003} proved the conjecture for quasiprimitive permutation groups. Consequently, any potential counterexample must have an intransitive minimal normal subgroup. Among permutation groups $G$ with an intransitive minimal normal subgroup $N$, there are known examples that are elusive~\cite{CGJKKMN2002,FKS1981,Giudici2007}. However, in each of these examples, $G$ is a split extension of $N$ by an elusive group, and Theorem~\ref{ThmE}, a generalization of a result of Wielandt, shows that they are not counterexamples to the Polycirculant Conjecture.

Let $\mathcal{E}$ denote the set of degrees of elusive groups. Considerable attention has been devoted in the literature to  $\mathcal{E}$. It is known~\cite{CGJKKMN2002,Giudici2007} that $\mathcal{E}$ contains
\begin{equation}\label{Eqn0}
2^2\cdot 3 \cdot 7^k \ \text{ and }\ 2^n\cdot p^k\cdot q_1\cdots q_r,
\end{equation}
where $p$ is a Mersenne prime, $k$ is a positive integer, $n$ is an integer such that $2^n>p$, each $q_i$ is a prime power not coprime to $p-1$, and $r\geq0$.
Observe that, if $G_1$ and $G_2$ are elusive groups on $\Omega_1$ and $\Omega_2$ respectively, then $G_1\times G_2$ is an elusive group on $\Omega_1\times\Omega_2$. It follows that $\mathcal{E}$ is closed under multiplication, and hence all products of numbers in~\eqref{Eqn0} also lie in $\mathcal{E}$. Nevertheless, these examples yield only a small set of known degrees for elusive groups: in particular, every such number is divisible by both $4$ and a Mersenne prime. This motivates the following open question, posed in~\cite{CGJKKMN2002}:

\begin{question}\label{QueDensity}
Does $\mathcal{E}$ have density zero in $\mathbb{N}$?
\end{question}

To study both Question~\ref{QueDensity} and the Polycirculant Conjecture, it is essential to broaden the constructions of elusive groups, especially to include non-split group extensions. Surprisingly, there is a key advantage in constructing elusive groups via non-split extensions: The permutation group $G$ needs not be elusive to construct an elusive group as a non-split extension of some group $N$ by $G$, since the nonexistence of derangements of certain prime orders in $G$ can possibly be ``overcome'' by the  nonexistence of elements of such prime orders in $(N.G)\setminus N$. This idea is illustrated in detail in Section~\ref{SecIdea} and enables us to construct new numbers in $\mathcal{E}$. Before giving these results, we note that our notation for group extensions follows the Atlas~\cite{CCNPW1985}. In particular, $N{:}G$ and $N{}^{\boldsymbol{\cdot}}G$ represent a split and non-split extension, respectively, of $N$ by $G$.

\begin{theorem}\label{ThmExample3}
For each Mersenne prime $p\geq7$ and integer $k\geq2$, there exists an elusive group of degree $p^{3k-4}(p+1)/2$ which has the form
\[
\big(\underbrace{\C_p^3.\C_p^3.\cdots.\C_p^3}_{k-1\textup{ copies of }\C_p^3}\big){}^{\boldsymbol{\cdot}}\PSL_2(p)
\]
and stabilizer $\C_p^2{:}\D_{p-1}$.
\end{theorem}

\begin{remark}
The group $\C_p^3.\C_p^3.\cdots.\C_p^3$ in Theorem~\ref{ThmExample3} is nonabelian for $k\geq3$, as shown in the remark after the proof of Lemma~\ref{LemN}.
\end{remark}

The next result, to the best of our knowledge, gives the first examples of elusive groups of odd degree and twice odd degree.

\begin{theorem}\label{ThmExample2}
For every integer $k\geq 1$, there exists an elusive group of degree $3^{k + 1}\cdot5^2 $ that has the form $(\C_{3^k}^4\times\C_5^3){}^{\boldsymbol{\cdot}}\A_5$ and stabilizer $(\C_{3^k}^3\times\C_5^2){:}\C_2^2$, and there exists an elusive group of degree $2 \cdot 3^{k + 1} \cdot 5^2 $ that has the form $(\C_{3^k}^4\times\C_5^3){}^{\boldsymbol{\cdot}}\A_5$ and stabilizer $(\C_{3^k}^3\times\C_5^2){:}\C_2$.
\end{theorem}

After illustrating the idea of constructing elusive groups $N{}^{\boldsymbol{\cdot}}G$ from non-elusive groups $G$ in Section~\ref{SecIdea}, we carry out this idea in Sections~\ref{SecExmp3} and~\ref{SecExmp2}, respectively, to construct elusive groups needed to prove Theorems~\ref{ThmExample3} and~\ref{ThmExample2}. The constructed elusive groups satisfy the conditions of Theorem~\ref{ThmE} and hence are not counterexamples to the Polycirculant Conjecture (see Section~\ref{SecRmk}), although they yield degrees not appearing in~\eqref{Eqn0}. Note that, by the direct product construction of elusive groups, products of these new degrees also belong to $\mathcal{E}$, thereby producing new numbers in $\mathcal{E}$.

We also emphasise that the primary goal of this paper is to initiate the construction of elusive groups as non-split extensions and to demonstrate the advantages of this approach, rather than endeavor to provide new degrees or examples. Our method seems promising to produce more elusive groups, and we believe that this will shed light into the study of elusive groups and the Polycirculant Conjecture. To conclude the paper, we pose some natural problems in Section~\ref{SecRmk}.

\section{Illustration of idea}\label{SecIdea}

For a prime $p$, a transitive permutation group is \emph{$p$-elusive} if it has no derangements of order $p$, and is \emph{$p'$-elusive} if it is $r$-elusive for each prime $r\neq p$.

\begin{lemma}\label{LemExt}
Let $G$ be a $p$-elusive group for some prime $p$ with point stabilizer $H$, and let $X=N.G$ be a transitive permutation group with point stabilizer $Y$ such that $|N|$ is coprime to $p$ and $\overline{Y}=H$, where $\overline{\phantom{w}}\colon X\to G$ is the quotient map modulo $N$. Then $X$ is $p$-elusive.
\end{lemma}

\begin{proof}
Let $a$ be an element (if any) of order $p$ in $X$. Since $|N|$ is not divisible by $p$, the order of $\overline{a}$ is also $p$.
Since $G$ is $p$-elusive with point stabilizer $H$, there exists $g\in G$ such that $\overline{a}^g\in H$.
This implies the existence of $b\in X$ such that $\overline{a^b}\in H=\overline{Y}$, that is, $a^b\in NY$.
Again, since $|N|$ is not divisible by $p$, every Sylow $p$-subgroup $P$ of $Y$ is a Sylow $p$-subgroup of $NY$.
Hence there exists $c\in NY$ such that $a^{bc}\in P\leq Y$. This shows that $X$ is $p$-elusive.
\end{proof}

By Lemma~\ref{LemExt}, to construct elusive groups as group extensions $N.G$ from a transitive permutation group $G$, we need only ensure that $N.G$ is $p$-elusive for every prime $p$ satisfying one of the following:
\begin{enumerate}[(i)]
\item $G$ is not $p$-elusive;
\item $p$ divides $|N|$.
\end{enumerate}
The likelihood of constructing elusive groups in this way increases when the number of primes $p$ satisfying one of~(i) or~(ii) is small. In fact, the successful constructions for Theorems~\ref{ThmExample3} and~\ref{ThmExample2} occur when this number is one and two, respectively.

\begin{example}\label{Exmp1}
Let $p\geq7$ be a Mersenne prime, and let $G\cong\PSL_2(p)$ be a transitive permutation group with point stabilizer $H\cong\D_{p-1}$, the dihedral group of order $p-1$. Since the prime divisors of $|G|/|H|=p(p+1)/2$ are $p$ and $2$, each Sylow $r$-subgroup of $G$ with $r\notin\{p,2\}$ is conjugate to a subgroup of $H$. Thus each element of prime order $r\notin\{p,2\}$ in $G$ is contained in a conjugate of $H$, which means that $G$ has no derangements of prime order $r\notin\{p,2\}$. Moreover, since $G\cong\PSL_2(p)$ has a unique conjugacy class of involutions and $|H|$ is divisible by $2$, it follows that $G$ has no derangements of prime order $r$ whenever $r\neq p$. In other words, $G$ is $p'$-elusive. Note that $G$ is not elusive as $p$ divides $|G|$ but not $|H|$.
\qed
\end{example}

Assume that $G$ is a $p'$-elusive group. If $G$ is not elusive, which means that condition~(i) holds uniquely for the prime $p$, then for $N.G$ to be elusive, the extension $N.G$ must be non-split.

\begin{lemma}\label{LemNonSplit}
Let $G$ be a transitive group with point stabilizer $H$ such that $G$ is not $p$-elusive for some prime $p$, and let $X$ be a transitive permutation group with point stabilizer $Y$ such that $X=N.G$ for some normal subgroup $N$ and that $\overline{Y}=H$, where $\overline{\phantom{w}}\colon X\to G$ is the quotient modulo $N$. If $X$ is $p$-elusive, then $X=N^{\boldsymbol{\cdot}}G$.
\end{lemma}

\begin{proof}
Suppose for a contradiction that $X=N{:}G$ is $p$-elusive. Since $G$ is not $p$-elusive, there exists some $g\in G$ with $|g|=p$ such that no conjugates of $g$ in $G$ lie in $H$. However, since $g\in G\leq N{:}G=X$ while $X$ is $p$-elusive, there exists $x\in X$ such that $g^x\in Y$. It follows that $\overline{g}^{\overline{x}}\in\overline{Y}=H$ with $\overline{g}=g$ and $\overline{x}\in G$, contradicting that no conjugates of $g$ in $G$ lie in $H$.
\end{proof}

In the setting of Lemma~\ref{LemNonSplit}, it is still possible that $N.G$ is elusive when $|N|$ is divisible by $p$.
For a prime $p$ and transitive permutation groups $G$ and $N.G$, we say that $N.G$ \emph{overcomes $p$ for $G$} if $N.G$ is $p$-elusive while $G$ is not.
Let us illustrate the possibility of such a situation by proving the following special case of Theorem~\ref{ThmExample3}.

\begin{proposition}\label{ThmExample1}
For each Mersenne prime $p\geq7$, there exists an elusive group of degree $p^2(p+1)/2$ which has the form $\C_p^3{}^{\boldsymbol{\cdot}}\PSL_2(p)$ and stabilizer $\C_p^2{:}\D_{p-1}$.
\end{proposition}

Let $G\cong\PSL_2(p)$ for some Mersenne prime $p\geq7$, and let $H\cong\D_{p-1}$ be a subgroup of $G$, as in Example~\ref{Exmp1}. Then $G$ is $p'$-elusive with stabilizer $H$. Consider the action of $G$ on $N\coloneqq\C_p^3$ by viewing $G=\Omega_3(p)$ acting on the space $N=\bbF_p^3$. Then $H$ is $\Omega_2^+(p){:}\C_2$ and stabilizes a $2$-dimensional subspace $M$ of $N$. Let $X=N.G=\C_p^3.\PSL_2(p)$ be an extension of $N$ by $G$ with this action of $G$, and let $Y=M.H=M{:}H=\C_p^2{:}\D_{p-1}$ be a subgroup of $\Nor_X(M)=N.H=N{:}H$ (note that $|N|$ and $|H|$ are coprime).

\begin{lemma}\label{LemHyp}
Let $X=N.G$ and $Y=M{:}H$ be as above. Then $Y$ is core-free in $X$. Moreover, the following are equivalent:
\begin{enumerate}[{\rm(a)}]
\item\label{ItemHyp1} $X=N^{\boldsymbol{\cdot}}G$;
\item\label{ItemHyp2} $X\setminus N$ has no elements of order $p$;
\item\label{ItemHyp3} the transitive permutation group $X$ with stabilizer $Y$ is elusive.
\end{enumerate}
\end{lemma}

\begin{proof}
Since the action of $G=\Omega_3(p)$ on $N=\C_p^3$ is irreducible, $N$ is the unique minimal normal subgroup of $X$, and so $Y$ is core-free in $X$.

\eqref{ItemHyp1}$\Rightarrow$\eqref{ItemHyp2}: Suppose for a contradiction that $X\setminus N$ has an element $x$ of order $p$. Since $|X|_p=|N|_p|G|_p=p^4$, the group $X$ has a Sylow $p$-subgroup $P:=N\langle x\rangle=N{:}\langle x\rangle$. Thus, by a theorem of Gasch\"{u}tz (see also~\cite[Main~Theorem~I.17.4]{Huppert2025}), $N$ has a complement in $X$, contradicting~\eqref{ItemHyp1}.

\eqref{ItemHyp2}$\Rightarrow$\eqref{ItemHyp3}:
Since $G$ is $p'$-elusive and $|N|=p^3$, we conclude by Lemma~\ref{LemExt} that $X$ is $p'$-elusive, and so it remains to prove that $X$ is $p$-elusive. Let $x$ be an element of order $p$ in $X$. Then $x\in N$ by~\eqref{ItemHyp2}. Identify $N$ with the orthogonal space $\mathbb{F}_p^3$ that $G=\Omega_3(p)$ acts on. In this way, $H=\Omega_2^+(p){:}\C_2$ stabilizes the $2$-subspace $M$. It follows that $M$ is a non-degenerate $2$-subspace of plus type and hence contains vectors of norm $\mu$ for each $\mu\in\mathbb{F}_p$. We conclude from Witt's Lemma that $x$, as a vector in $N$, is mapped into $M$ by some element of $G$. This means that $x$ is conjugate in $X$ to some element of $M<Y$. So $X$ is $p$-elusive, as desired.

\eqref{ItemHyp3}$\Rightarrow$\eqref{ItemHyp1}:
This follows immediately from Lemma~\ref{LemNonSplit}.
\end{proof}

Note that the natural action of $\Omega_3(p)$ on $\mathbb{F}_p^3$ is the unique irreducible representation of $\Omega_3(p)\cong\PSL_2(p)$ on $\mathbb{F}_p^3$, which is isomorphic to the symmetric square of the natural module $\mathbb{F}_p^2$ of $\SL_2(p)$ (see~\cite[Proposition~5.4.11]{KL1990}). In the notation of~\cite{UoGVAG2012}, this is the module $L(2\omega_1)$ with $\omega_1$ being the only fundamental dominant weight for $A_1(\bbF_p)\cong\PSL_2(p)$, and it is shown in~\cite[Theorem~1.2.5]{UoGVAG2012} that $\mathrm{H}^2(A_1(\bbF_p),L(2\omega_1))$ has dimension $1$.

\begin{proof}[\rm\textbf{Proof of Proposition~\ref{ThmExample1}}]
Let $G=\Omega_3(p)\cong\PSL_2(p)$ for some Mersenne prime $p\geq7$, and let $N=\C_p^3$, viewed also as the vector space $\bbF_p^3$ that $G$ naturally acts on. Since $\mathrm{H}^2(G,N)$ has dimension $1$, there exists a non-split extension $X=N^{\boldsymbol{\cdot}}G$ as in~\eqref{ItemHyp1} of Lemma~\ref{LemHyp}, and so there exists an elusive group $X=N^{\boldsymbol{\cdot}}G=\C_p^3{}^{\boldsymbol{\cdot}}\PSL_2(p)$ with stabilizer $Y=\C_p^2{:}\D_{p-1}$.
\end{proof}

In the remainder of the section, we give a more precise construction of the elusive groups $\C_p^3{}^{\boldsymbol{\cdot}}\PSL_2(p)$ asserted by Proposition~\ref{ThmExample1}. This not only supplements the above proof of Proposition~\ref{ThmExample1}, but also presents the idea used in Section~\ref{SecExmp2}.

\begin{construction}\label{Exmp3}
Let $p\geq7$ be a Mersenne prime, and let
\begin{align*}
X=\langle&a,b,c,s,t\mid a^p=b^p=c^p=[a,b]=[b,c]=[c,a]=t^2=(st)^3=(s^2ts^\frac{p+1}{2}t)^3=1,\\
&s^p=a,\, a^s=a,\,b^s=ab,\,c^s=a^{-1}b^{-2}c,\,a^t=c^{-1},\,b^t=b^{-1},\,c^t=a^{-1}\rangle.
\end{align*}
It is clear that $N:=\langle a,b,c\rangle$ is a normal subgroup of $X$. Let $\,\overline{\phantom{\varphi}}\,$ be the quotient map from $X$ to $G:=X/N$.
\qed
\end{construction}

We will see that $X$ and $N$ in the above construction satisfy $X/N\cong\PSL_2(p)$ with the action of $X/N$ on $N$ equivalent to that of $\Omega_3(p)$ on $\bbF_p^3$. Moreover, we will show that $X\setminus N$ has no elements of order $p$, which enables us to conclude by Lemma~\ref{LemHyp} that $X$ is an elusive group with stabilizer $Y$ described in the lemma. We will work within a Sylow $p$-subgroup of $X$ to prove that $X\setminus N$ has no elements of order $p$. Consider
\begin{align}\label{Eqn4}
W=\langle&a,b,c,s\mid a^p=b^p=c^p=[a,b]=[b,c]=[c,a]=1,\nonumber\\
&s^p=a,\,a^s=a,\,b^s=ab,\,c^s=a^{-1}b^{-2}c\rangle.
\end{align}
It is straightforward to verify that the correspondences
\[
a\mapsto x^p,\ \ b\mapsto x^{\frac{p(p-1)}{2}}y^{-1},\ \ c\mapsto z^{-2},\ \ s\mapsto x
\]
induce an isomorphism (with inverse induced by $x\mapsto s$, $y\mapsto a^{(p-1)/2}b^{-1}$ and $z\mapsto c^{(p-1)/2}$) from $W$ to
\begin{equation}\label{Eqn5}
\langle x,y,z\mid x^{p^2}=y^p=z^p=[y,z]=1,\,[x,y]=x^p,\,[x,z]=y\rangle.
\end{equation}
The group~\eqref{Eqn5} is III(xi) in the table of groups of order $p^n$ in \cite[\S73]{Burnside1897}. In particular, it has order $p^4$, and so does $W$.

\begin{lemma}\label{lem-nop}
Let $W$ be the group defined in~\eqref{Eqn4} for an odd prime $p$. Then $W\setminus \langle a,b,c\rangle$ has no elements of order $p$.
\end{lemma}

\begin{proof}
Suppose for a contradiction that $W\setminus \langle a,b,c\rangle$ has an element of order $p$. Let $ds^i$ be such an element, where $d\in \langle a,b,c\rangle$ and $i\in\mathbb{Z}$. According to the Hall-Petrescu identity (see~\cite[Theorem~III.9.4]{Huppert2025}),
\[
d^p(s^i)^p=(ds^i)^pc_2^{\binom{p}{2}}\cdots c_{p-1}^{\binom{p}{p-1}} c_p
\]
for some $c_2,\ldots,c_p$ such that, for $i\in\{2,\ldots,p\}$, the element $c_i$ lies in the $i$-th term $W_i$ of the lower central series of $W$. Since $|W|$ divides $p^4$, the nilpotency class of $W$ is at most $3$ (indeed, the nilpotency class of $W$ is precisely $3$, but we do not need this fact here), and so $c_4=\dots=c_p=1$. Note that $W_3\leq W_2=W'\leq \langle a,b,c\rangle$. It follows from $(ds^i)^p=1$ that
\[
a^i=(s^i)^p=d^p(s^i)^p=c_2^{\binom{p}{2}}c_3^{\binom{p}{3}}=1,
\]
which implies that $i$ is divisible by $p$. However, this leads to $ds^i=da^{i/p}\in\langle a,b,c\rangle$, contradicting our assumption.
\end{proof}

Now let $X$, $N$, $G$ and $\,\overline{\phantom{\varphi}}\,$ be as in Construction~\ref{Exmp3}. Denote $S=\overline{s}$ and $T=\overline{t}$. Then
\[
G=\overline{X}=\langle S,T\mid S^p=T^2=(ST)^3=(S^2TS^\frac{p+1}{2}T)^3=1\rangle,
\]
and so we conclude by~\cite[Theorem~A]{BM1968} that $G\cong\PSL_2(p)\cong\Omega_3(p)$.
With $N$ viewed as the vector space $\bbF_p^3$, the action $\psi$ of $G=X/N$ on $N$ is given by
\[
S^\psi=
\begin{pmatrix}
1&0&0\\
1&1&0\\
-1&-2&1
\end{pmatrix}
\ \text{ and }\ T^\psi=
\begin{pmatrix}
0&0&-1\\
0&-1&0\\
-1&0&0
\end{pmatrix}.
\]
Let $Q$ be the quadratic form on $N$ such that
\[
Q(\alpha,\beta,\gamma)=4\alpha\gamma+\beta^2
\]
for all $(\alpha,\beta,\gamma)\in\bbF_p^3=N$. It is straightforward to verify that $S^\psi$ and $T^\psi$ preserve $Q$, so $\langle S^\psi,T^\psi\rangle\leq\mathrm{GO}(N,Q)$, the general orthogonal group acting on the space $N$ and preserving the quadratic form $Q$. Since $G=\langle S,T\rangle\cong\Omega_3(p)$,
\[
\langle S^\psi,T^\psi\rangle=\Omega(N,Q)=\Omega_3(p).
\]
In other words, the action of $G$ on $N$ is that of $\Omega_3(p)$ on $\bbF_p^3$.
Moreover, as the Sylow $p$-subgroups of $G\cong\PSL_2(p)$ have order $p$, the group $W$ is a Sylow $p$-subgroup of $X$. Lemma~\ref{lem-nop} implies that $X\setminus N$ has no elements of order $p$, and so by Lemma~\ref{LemHyp}, $X$ is an elusive group with stabilizer $Y$ described there.

\section{Elusive groups of degree $p^{3k-4}(p+1)/2$}\label{SecExmp3}

Recall that, for a prime number $p$ and positive integer $k$, the residue class ring $\mathbb{Z}/ p^k\mathbb{Z}$ of integers modulo $p^k$ has multiplicative group $(\mathbb{Z}/p^k\mathbb{Z})^{\times}\cong\C_{p^{k-1}(p-1)}$. The two-dimensional special linear group $\SL_2(\mathbb{Z}/p^k\mathbb{Z})$ over $\mathbb{Z}/p^k\mathbb{Z}$ is defined as
\[
\SL_2(\mathbb{Z}/p^k\mathbb{Z})=\left\{\left.
\begin{pmatrix}
 a&b\\ c&d
\end{pmatrix}
\,\right|\,  a,b,c,d \in \mathbb{Z}/ p^k\mathbb{Z},\,ad-bc=1 \right\}.
\]
It is well known and fairly easy to check that the modular group $\SL_2(\mathbb{Z})$ is generated by
\[
\begin{pmatrix}
1&1\\0&1
\end{pmatrix}
\ \text{ and }\
\begin{pmatrix}
 0&1\\-1&0
\end{pmatrix}
\]
(see~\cite[Exercise~1.1.1]{DS2005}), and so is $\SL_2(p)$ (see~\cite[Theorem~2.8.4]{Gorenstein1980}). This implies that the natural ring epimorphism $\mathbb{Z}\to\mathbb{F}_p$ induces an epimorphism from  $\SL_2(\mathbb{Z})\to \SL_2(p)$. As a consequence, the homomorphism
\[
\varphi\colon\SL_2(\mathbb{Z}/p^k\mathbb{Z})\to\SL_2(p)
\]
induced by the natural ring epimorphism $\mathbb{Z}/p^k\mathbb{Z}\to\mathbb{F}_p$ is also an epimorphism. Denote by $I$ the $2\times2$ identity matrix (over any field).

\begin{construction}\label{Cons1}
Let $R=\mathbb{Z}/p^k\mathbb{Z}$ be the residue class ring of integers modulo $p^k$ for some odd prime $p$ and integer $k\ge 2$, and let $\widehat{X}=\SL_2(R)$. Take an element $w$ of order $p-1$ in $R^{\times}$. Let
\begin{align*}
P&=\left\{ \left.I+p^{k-1}
\begin{pmatrix}
 0&b\\c&0
\end{pmatrix}
\,\right|\, b,c\in R \right\}\cong\C_p^2,\\
Q&=\left\langle
\begin{pmatrix}
 w&0\\0&w^{-1}
\end{pmatrix},
\begin{pmatrix}
 0&1\\-1&0
\end{pmatrix}
\right\rangle
=
\left\langle
\begin{pmatrix}
 w&0\\0&w^{-1}
\end{pmatrix}
\right\rangle
{:}
\left\langle
\begin{pmatrix}
 0&1\\-1&0
\end{pmatrix}
\right\rangle
\cong\Q_{2(p-1)},
\end{align*}
and $\widehat{Y}=\langle P,Q\rangle$. Then $\widehat{Y}=P{:}Q\cong\C_p^2{:}\Q_{2(p-1)}$.
Let $\varphi\colon\widehat{X}=\SL_2(R)\to\SL_2(p)$ be the epimorphism induced by the natural ring epimorphism $R\to\mathbb{F}_p$,
let $\widehat{N}=\ker(\varphi)$, and let $Z$ be the subgroup of $\widehat{X}$ consisting of $\pm I$.
Clearly, $Z\cong\C_2$ is normal in $\widehat{X}$. Let $X$, $Y$ and $N$ be the images of $\widehat{X}$, $\widehat{Y}$ and $\widehat{N}$ under the quotient map $\widehat{X}\to\widehat{X}/Z$; that is, $X=\widehat{X}/Z$, $Y=\widehat{Y}/Z$ and $N=\widehat{N}Z/Z$.
\qed
\end{construction}

For $2\times2$ matrices $A=(a_{i,j})$ and $B=(b_{i,j})$ over $\mathbb{Z}/p^k\mathbb{Z}$ and for $\ell\in\{1,\ldots,k\}$, denote
\[
A\equiv B\ \,(\bmod\,p^\ell)
\]
if $a_{i,j}\equiv b_{i,j}\ (\bmod\,p^\ell)$ for each $i,j\in\{1,2\}$. While the following lemma is closely related to results in~\cite[III\S17]{Huppert2025}, we provide a self-contained treatment here to establish the specific properties required for our analysis.

\begin{lemma}\label{LemN}
With the notation in Construction~$\ref{Cons1}$,
\[
N\cong\widehat{N}=\underbrace{\C_p^3.\C_p^3.\cdots.\C_p^3}_{k-1\textup{ copies of }\C_p^3}
\]
and $Y\cap N=PZ/Z\cong\C_p^2$.
\end{lemma}

\begin{proof}
For $\ell\in\{1,\ldots,k\}$, consider the homomorphism $\widehat{X}=\SL_2(R)\to\SL_2(\mathbb{Z}/p^\ell\mathbb{Z})$ induced by the natural ring epimorphism $R=\mathbb{Z}/p^k\mathbb{Z}\to\mathbb{Z}/p^\ell\mathbb{Z}$, and let $N_\ell$ be its kernel. Then
\[
N_\ell=\{A\in\SL_2(R)\mid A\equiv I\ (\bmod\,p^\ell)\}
\]
is a normal subgroup of $\widehat{X}$, and we have a normal series
\begin{equation}\label{Eqn1}
\widehat{N}=N_1\vartriangleright N_2\vartriangleright\cdots\vartriangleright N_{k-1}\vartriangleright N_k=1.
\end{equation}

Fix $\ell\in\{1,\ldots,k-1\}$, and take elements
\[
A_\ell=
\begin{pmatrix}
 1+p^\ell&p^\ell\\-p^\ell&1-p^\ell
\end{pmatrix},\ \
B_\ell=
\begin{pmatrix}
 1&p^\ell\\0&1
\end{pmatrix},\ \
C_\ell=
\begin{pmatrix}
 1&0\\p^\ell&1
\end{pmatrix}
\]
in $\SL_2(R)$. Clearly, $A_\ell,B_\ell,C_\ell\in N_\ell$. Let $\varphi_\ell$ be the quotient map $N_\ell\to N_\ell/N_{\ell+1}$. Since
\begin{align*}
A_\ell B_\ell&\equiv
\begin{pmatrix}
 1+p^\ell&2p^\ell\\-p^\ell&1-p^\ell
\end{pmatrix}
\equiv B_\ell A_\ell\ \,(\bmod\,p^{\ell+1}),\\
A_\ell C_\ell&\equiv
\begin{pmatrix}
 1+p^\ell&p^\ell\\0&1-p^\ell
\end{pmatrix}
\equiv C_\ell A_\ell\ \,(\bmod\,p^{\ell+1}),\\
B_\ell C_\ell&\equiv
\begin{pmatrix}
 1&p^\ell\\p^\ell&1
\end{pmatrix}
\equiv C_\ell B_\ell\ \,(\bmod\,p^{\ell+1}),
\end{align*}
the subgroup $\langle A_\ell^{\varphi_\ell},B_\ell^{\varphi_\ell},C_\ell^{\varphi_\ell}\rangle$ of $N_\ell/N_{\ell+1}$ is abelian. Note
\[
A_\ell=I+p^\ell
\begin{pmatrix}
 1&1\\-1&-1
\end{pmatrix},\ \
B_\ell=I+p^\ell
\begin{pmatrix}
 0&1\\0&0
\end{pmatrix},\ \
C_\ell=I+p^\ell
\begin{pmatrix}
 0&0\\1&0
\end{pmatrix}.
\]
We conclude that $\langle A_\ell^{\varphi_\ell},B_\ell^{\varphi_\ell},C_\ell^{\varphi_\ell}\rangle=\langle A_\ell^{\varphi_\ell}\rangle\times\langle B_\ell^{\varphi_\ell}\rangle\times\langle C_\ell^{\varphi_\ell}\rangle\cong\C_p^3$.
For each $A\in N_\ell$, since $A\equiv I\ (\bmod\,p^\ell)$,
\[
A=I+p^\ell
\begin{pmatrix}
 a&b\\c&d
\end{pmatrix}
\]
for some $a,b,c,d\in R$. Since $\det(A)=1$, it follows that $p^\ell(a+d)=p^{2\ell}(bc-ad)$, and so $a+d\equiv0\ (\bmod\,p)$. As a consequence,
\[
A\equiv I+p^\ell
\begin{pmatrix}
 a&b\\c&-a
\end{pmatrix}
\ \,(\bmod\,p^{\ell+1}).
\]
This together with the observation
\begin{align*}
A_\ell^aB_\ell^{b-a}C_\ell^{c+a}&\equiv
\left(I+p^\ell a
\begin{pmatrix}
 1&1\\-1&-1
\end{pmatrix}
\right)
\left(I+p^\ell(b-a)
\begin{pmatrix}
 0&1\\0&0
\end{pmatrix}
\right)
\left(I+p^\ell(c+a)
\begin{pmatrix}
 0&0\\1&0
\end{pmatrix}
\right)\\
&\equiv I+p^\ell\left(a
\begin{pmatrix}
 1&1\\-1&-1
\end{pmatrix}
+(b-a)
\begin{pmatrix}
 0&1\\0&0
\end{pmatrix}
+(c+a)
\begin{pmatrix}
 0&0\\1&0
\end{pmatrix}\right)\\
&\equiv I+p^\ell
\begin{pmatrix}
 a&b\\c&-a
\end{pmatrix}
\ \,(\bmod\,p^{\ell+1})
\end{align*}
implies that $A^{\varphi_\ell}=(A_\ell^{\varphi_\ell})^a(B_\ell^{\varphi_\ell})^{b-a}(C_\ell^{\varphi_\ell})^{c+a}\in\langle A_\ell^{\varphi_\ell},B_\ell^{\varphi_\ell},C_\ell^{\varphi_\ell}\rangle$. Therefore,
\begin{equation}\label{Eqn2}
N_\ell/N_{\ell+1}=\langle A_\ell^{\varphi_\ell},B_\ell^{\varphi_\ell},C_\ell^{\varphi_\ell}\rangle\cong\C_p^3.
\end{equation}

Combining~\eqref{Eqn1} and~\eqref{Eqn2}, we obtain
\[
\widehat{N}=(N_{k-1}/N_k).(N_{k-2}/N_{k-1}).\cdots.(N_1/N_2)=\underbrace{\C_p^3.\C_p^3.\cdots.\C_p^3}_{k-1\textup{ copies of }\C_p^3}.
\]
In particular, $|\widehat{N}|$ is odd and hence coprime to $|Z|$. Consequently, $N=\widehat{N}Z/Z\cong\widehat{N}$.

Finally, it follows from $P\leq\widehat{N}$ and $Q\cap\widehat{N}=1$ that $\widehat{Y}\cap\widehat{N}=(PQ)\cap\widehat{N}=P$. This leads to $\widehat{Y}\cap(\widehat{N}Z)=(\widehat{Y}\cap\widehat{N})Z=PZ$ and hence
\[
Y\cap N=(\widehat{Y}/Z)\cap(\widehat{N}Z/Z)=\big(\widehat{Y}\cap(\widehat{N}Z)\big)/Z=PZ/Z\cong P\cong\C_p^2.\qedhere
\]
\end{proof}

\begin{remark}
For $k\geq3$, the group $\widehat{N}$ is nonabelian. This can be seen, for example, from the following calculation:
\begin{align*}
\begin{pmatrix}1+p& p\\ -p & 1-p\end{pmatrix}\begin{pmatrix}1&p\\0&1\end{pmatrix}&=\begin{pmatrix}1+p&2p+p^2\\-p&1-p-p^2 \end{pmatrix}\\
\begin{pmatrix}1&p\\0&1\end{pmatrix}\begin{pmatrix}1+p& p\\ -p & 1-p\end{pmatrix}&=\begin{pmatrix}1+p-p^2&2p-p^2\\-p&1-p \end{pmatrix}.
\end{align*}
\end{remark}

\begin{lemma}\label{lem-sl2r-cons}
With the notation in Construction~$\ref{Cons1}$, $Y$ is core-free in $X$, and the transitive permutation group $X$ with stabilizer $Y$ is elusive.
\end{lemma}

\begin{proof}
Note that the core of $\widehat{Y}$ in $\SL_2(R)$ is $Z$ (this can be either directly verified or deduced from the classification of normal subgroups of $\SL_2(R)$ in~\cite{McQuillan1965}). We conclude that $Y=\widehat{Y}/Z$ is core-free in $\SL_2(R)/Z=X$. Hence $X$ is a transitive permutation group with stabilizer $Y$, and is $p'$-elusive by Lemma~\ref{LemExt}. To prove that it is elusive, we verify that every element of order $p$ in $X$ is conjugate to some element in $Y$. Since $\gcd(|Z|,p)=\gcd(2,p)=1$, we only need to show that every element of order $p$ in $\widehat{X}$ is conjugate to some element in $\widehat{Y}$.

Take an arbitrary element $A$ of order $p$ in $\widehat{X}=\SL_2(R)$. We first prove
\begin{equation}\label{Eqn3}
A^\varphi\neq
\begin{pmatrix}
 1&1\\0&1
\end{pmatrix}.
\end{equation}
Suppose that this is not the case. Then
\[
A=
\begin{pmatrix}
1&1\\0&1
\end{pmatrix}
+p^{k-1}B=I+E+p^{k-1}B
\]
for some $2\times 2$ matrix $B$ over $R$, where
\[
E:=\begin{pmatrix}
 0&1\\0&0
\end{pmatrix}.
\]
Since $E^2$ is the zero matrix,
\begin{align*}
  A^p&=(I+E+p^{k-1}B)^p\\
  &=(I+E)^p+p^{k-1}\sum_{i=0}^{p-1}(I+E)^iB(I+E)^{p-1-i}\\
  &=I+pE+p^{k-1}\sum_{i=0}^{p-1}(I+iE)B(I+(p-1-i)E)\\
  &=I+pE+p^{k-1}\sum_{i=0}^{p-1}(B+iEB+(p-1-i)BE)\\
  &=I+pE+\left(p^{k-1}\sum_{i=0}^{p-1}1\right)B+\left(p^{k-1}\sum_{i=0}^{p-1}i\right)EB+\left(p^{k-1}\sum_{i=0}^{p-1}(p-1-i)\right)BE\\
  &=I+pE,
\end{align*}
contradicting $A^p=I$. Now, if $A$ is an element of order $p$ in $\SL_2(R)$, then $A^\varphi\neq I+E$ as~\eqref{Eqn3} asserts, from which we derive $A^\varphi\notin\langle I+E\rangle$ by replacing $A$ with powers of $A$. Based on this, since $\langle I+E\rangle$ is a Sylow $p$-subgroup of $\SL_2(p)$, we conclude that $A^\varphi$ cannot be an element of order $p$ in $\SL_2(p)$ (this is obtained by replacing $A$ with its conjugates in $\SL_2(R)$). Therefore, $A^\varphi=I$, and so
\[
A=
\begin{pmatrix}
 1+p^{k-1}a&p^{k-1}b\\p^{k-1}c&1+p^{k-1}d
\end{pmatrix}
\]
for some $a,b,c,d\in R$. Note that $\det(A)=1$ implies $p^{k-1}a+p^{k-1}d=0$.

When $p^{k-1}b\neq 0$, there exists some $e\in R$ such that $be=1$, and so with
\[
D:=
\begin{pmatrix}
1& 0\\-ae &1
\end{pmatrix}
\in\SL_2(R),
\]
we conclude that
\begin{align*}
D^{-1}AD&=
\begin{pmatrix}
1& 0\\ae &1
\end{pmatrix}
\begin{pmatrix}
 1+p^{k-1}a&p^{k-1}b\\p^{k-1}c&1-p^{k-1}a
\end{pmatrix}
\begin{pmatrix}
1& 0\\-ae &1
\end{pmatrix}\\
&=I+p^{k-1}
\begin{pmatrix}
0&b\\a^2e+c&0
\end{pmatrix}
\in P<\widehat{Y}.
\end{align*}
Similarly, when $p^{k-1}c\neq 0$, some element of $\SL_2(R)$ conjugates $A$ into $\widehat{Y}$. Now assume that $p^{k-1}b=p^{k-1}c=0$. Since $A\neq I$, it follows that $p^{k-1}d=-p^{k-1}a\neq 0$. Since
\begin{align*}
\begin{pmatrix}
 1&1\\0&1
\end{pmatrix}^{-1}A
\begin{pmatrix}
 1&1\\0&1
\end{pmatrix}
&=
\begin{pmatrix}
 1&-1\\0&1
\end{pmatrix}
\begin{pmatrix}
 1+p^{k-1}a&0\\0&1-p^{k-1}a
\end{pmatrix}
\begin{pmatrix}
 1&1\\0&1
\end{pmatrix}\\
&=
\begin{pmatrix}
 1+p^{k-1}a& 2p^{k-1}a\\0& 1-p^{k-1}a
\end{pmatrix}
\end{align*}
with $2p^{k-1}a\neq0$, we conclude by the previous cases that $A$ is conjugated into $\widehat{Y}$ by some element of $\SL_2(R)$.
This completes the proof.
\end{proof}

We are now in a position to derive Theorem~\ref{ThmExample3}.

\begin{proof}[\rm\textbf{Proof of Theorem~\ref{ThmExample3}}]
Follow the notation in Construction~\ref{Cons1}.
According to Lemma~\ref{LemN}, $|N|=p^{3k-3}$ and $Y\cap N=PZ/Z$. Since $\widehat{N}=\ker(\varphi)$ is normal in $\SL_2(R)$, we deduce that $N=\widehat{N}Z/Z$ is normal in $\SL_2(R)/Z=X$ where
\begin{align*}
X/N&\cong\SL_2(R)/(\widehat{N}Z)\cong\SL_2(R)^\varphi/(\widehat{N}Z)^\varphi=\SL_2(R)^\varphi/Z^\varphi=\PSL_2(p),\\
YN/N&\cong Y/(Y\cap N)=(\widehat{Y}/Z)/(PZ/Z)\cong\widehat{Y}/(PZ)=(PQ)/(PZ)\cong Q/Z\cong\D_{p-1}.
\end{align*}
The conclusion follows from Lemmas~\ref{LemN} and~\ref{lem-sl2r-cons}.
\end{proof}

\section{Elusive groups of degree not divisible by $4$}\label{SecExmp2}

Consider
\begin{equation}\label{Eqn6}
\A_5=\langle x=(1,2)(3,4),\,y=(1,3,5)\rangle\cong\langle x,y\mid x^2=y^3=(xy)^5=1\rangle.
\end{equation}
Its transitive representation of degree $15$ with stabilizer a Sylow 2-subgroup $\C_2^2$ is $2$-elusive, but neither $3$- nor $5$-elusive.
Let $U=\bbF_3^4$ and $V=\bbF_5^3$ be the fully deleted permutation modules of $\A_5$ over $\bbF_3$ and $\bbF_5$, respectively, which are both (absolutely) irreducible modules of $\A_5$~\cite[p.\ 186]{KL1990}. According to~\cite[Corollary~1]{KP1993I} and~\cite{KP1993II}, both $\mathrm{H}^2(\A_5,U)$ and $\mathrm{H}^2(\A_5,V)$ have dimension $1$. Hence, it may be possible to construct an extension of $\A_5$ that overcomes $3$ and $5$ for $\A_5$. In this section, we construct infinitely many such extensions, leading to a proof of Theorem~\ref{ThmExample2}.

If $\bigoplus_{i=1}^5\bbF_3e_i$ is the permutation module of $\A_5$ over $\bbF_3$, then take
\[
u_i=e_i-e_5\ \text{ for }\,i\in\{1,2,3,4\},
\]
so $U=\bigoplus_{i=1}^4\bbF_3u_i$. Then the action of $\A_5$ on $U$ (in multiplicative notation) is given by
\begin{align}
\label{Eqn7}&u_1^x=u_2,\ \ u_2^x=u_1,\ \ u_3^x=u_4,\ \ u_4^x=u_3,\\
\label{Eqn8}&u_1^y=u_1^{-1}u_3,\ \ u_2^y=u_1^{-1}u_2,\ \ u_3^y=u_1^{-1},\ \ u_4^y=u_1^{-1}u_4.
\end{align}
Similarly, if $\bigoplus_{i=1}^5\bbF_5e_i$ is the permutation module of $\A_5$ over $\bbF_5$, then take
\[
v_i=e_i-e_5+\bbF_5(e_1+e_2+e_3+e_4+e_5)\ \text{ for }\,i\in\{1,2,3\},
\]
so $V=\bigoplus_{i=1}^3\bbF_5v_i$ with
\begin{align}
\label{Eqn9}&v_1^x=v_2,\ \ v_2^x=v_1,\ \ v_3^x=v_1^{-1}v_2^{-1}v_3^{-1},\\
\label{Eqn10}&v_1^y=v_1^{-1}v_3,\ \ v_2^y=v_1^{-1}v_2,\ \ v_3^y=v_1^{-1}.
\end{align}
Inspired by~\eqref{Eqn6}--\eqref{Eqn10}, we construct as follows a group $X$ as an extension of $U\times V$ that can overcome $3$ and $5$ for $\A_5$. Fix an integer $k\geq 1$. Consider the group $X$ with generators $u_1$, $u_2$, $u_3$, $u_4$, $v_1$, $v_2$, $v_3$, $x$, $y$ and defining relations $u_i^{3^k}=[u_i,u_j]=v_i^5=[v_i,v_j]=[u_i,v_j]=1$ for all admissible $i$ and $j$,
\[
x^2=1,\ \ y^3=u_2u_4^{-1},\ \ (xy)^5=v_1^{-2}v_2^2v_3
\]
and those in~\eqref{Eqn7},~\eqref{Eqn8},~\eqref{Eqn9},~\eqref{Eqn10}. Let $U=\langle u_1,u_2,u_3,u_4\rangle$ and $V=\langle v_1,v_2,v_3\rangle$ and $N=\langle U,V\rangle$.

\begin{lemma}\label{GenXk}
For every integer $k\geq 1$, the group $X$ has the form $(\C_{3^k}^4\times \C_5^3).\A_5$.
\end{lemma}

\begin{proof}
Based on the relations of $X$ satisfied by the generators $u_1,u_2,u_3,u_4$ and $v_1,v_2,v_3$, it is clear that $U$ and $V$ are quotients of $\C_{3^k}^4$ and $\C_5^3$, respectively. In particular, $U$ and $V$ must intersect trivially. Moreover, both $U$ and $V$ are normal subgroups of $X$. It follows that $N = \langle U,V\rangle = UV = U\times V$.
Since $x^2,y^3,(xy)^5\in N$, the group $X/N = X/(U\times V)$ is isomorphic to $\A_5$ by~\eqref{Eqn6}. Hence, $X \cong (U\times V).\A_5$. Therefore, in order to prove the lemma, it suffices to show that $|U| = 3^{4k}$ and $|V| = 5^3$.

Let $X_U$ and $X_V$ denote the quotients of $X$ by $U$ and $V$, respectively, each admitting a presentation obtained from the presentation of $X$ by setting the generators $u_1,u_2,u_3,u_4$ and $v_1,v_2,v_3$ to $1$, respectively.
One can directly check (for example, by \textsc{Magma}~\cite{BCP1997}) that $X_U$ is isomorphic to ${\C_5^3}.\A_5$. As a consequence, $|V|=5^3$. To complete the proof, we show that $X_V$ is isomorphic to ${\C_{3^k}^4}.\A_5$. We do this by specifying an embedding $\phi$ of $X_V=\langle u_1,u_2,u_3,u_4,x,y\rangle$ into the wreath product $W\coloneqq\C_{3^{k+1}}\wr \A_5 = \C_{3^{k+1}}^5\rtimes \A_5$.

Let $b_1,\ldots,b_5$ be the natural generators of the base group $\C_{3^{k+1}}^5$ of $W$. For $i \in \{1,2,3,4\}$, let
\[
h_i= b_i^3 b_5^{-3}.
\]
Let $r = (1,2)(3,4)$ and $s = (1,3,5)$ be elements of $A_5$, and let
\[
\rho = b_3^{-1}b_4r\ \ \text{and}\ \ \sigma = b_2b_4^{-1}s
\]
be elements of $W$. We first verify that $\phi\colon X_V\rightarrow W$ given by $\phi(u_i)= h_i$ for $i\in\{1,2,3,4\}$, $\phi(x) = \rho$ and $\phi(y) = \sigma$ is a group homomorphism.

Obviously, $h_i^{3^k} = 1$ and $[h_i,h_j]=1$ for all $i,j\in\{1,2,3,4\}$. Observe that
\begin{align*}
\rho^2&=b_3^{-1}b_4(b_3^{-1}b_4)^{r}r^2 =  b_3^{-1}b_4b_4^{-1}b_3= 1,\\
\sigma^3&=(b_2b_4^{-1})^3 s^3 = b_2^3b_4^{-3} = b_2^3b_5^{-3} (b_4^3b_5^{-3})^{-1} = h_2h_4^{-1},
\end{align*}

so the relations $\rho^2 = 1$ and $\sigma^3 = h_2h_4^{-1}$ hold. To verify the relation $(\rho\sigma)^5 = 1$, we compute
\begin{align*}
\rho\sigma= b_3^{-1}b_4rb_2b_4^{-1}s &= b_3^{-1}b_4(b_2b_4^{-1})^r rs  \\
&= b_3^{-1}b_4 b_1b_3^{-1}(1,2)(3,4)(1,3,5)= b_1b_3^{-2}b_4(1,2,3,4,5)
\end{align*}
and write $\rho\sigma=bt$ with $b= b_1b_3^{-2}b_4$ and $t = (1,2,3,4,5)$. Then
\[
(\rho\sigma)^5 = (bt)^5 = bb^{t^{-1}}b^{t^{-2}}b^{t^{-3}}b^{t^{-4}}t^5 = bb^{t}b^{t^2}b^{t^3}b^{t^4},
\]
and since
\[
b^t = b_2b_4^{-2}b_5, \ \ b^{t^2} = b_3b_5^{-2}b_1, \ \ b^{t^3} = b_4b_1^{-2}b_2, \ \ b^{t^4} = b_5b_2^{-2}b_3,
\]
it follows that $(bt)^5 = 1$. In particular, $(\rho\sigma)^5 = 1$. Now we establish the compatibility of the generators $h_1,h_2,h_3,h_4,\rho,\sigma$ with relations~\eqref{Eqn7} and~\eqref{Eqn8}. Let $i\in \{1,2,3,4\}$. Then
\begin{align*}
        h_i^\rho = \rho h_i\rho= b_3^{-1}b_4rb_i^3b_5^{-3}b_3^{-1}b_4r
        &=b_3^{-1}b_4(b_i^3b_5^{-3}b_3^{-1}b_4)^{r}  \\
        &= b_3^{-1}b_4b_{i^r}^3b_5^{-3}b_4^{-1}b_3 = b_{i^r}^3b_5^{-3}= h_{i^r},
\end{align*}
as in~\eqref{Eqn7}. For the relation corresponding to~\eqref{Eqn8},
\begin{align*}
        h_i^\sigma = b_2^{-1}b_4s^{-1}b_i^3b_5^{-3}b_2b_4^{-1}s
        &= b_2^{-1}b_4(b_i^3b_5^{-3}b_2b_4^{-1})^s\\
        &= b_2^{-1}b_4 b_{i^s}^3 b_1^{-3}b_2b_4^{-1}=b_1^{-3}b_{i^s}^{3}= h_1^{-1}b_{i^s}^3b_5^{-3},
\end{align*}
from which it follows immediately that
\begin{align*}
h_1^\sigma = h_1^{-1}b_{3}^3b_5^{-3} = h_1^{-1}h_3&,\quad
h_2^\sigma = h_1^{-1}b_{2}^3b_5^{-3} = h_1^{-1}h_2,\\
h_3^\sigma = h_1^{-1}b_{5}^3b_5^{-3} = h_1^{-1}&,\quad
h_4^\sigma = h_1^{-1}b_{4}^3b_5^{-3} = h_1^{-1}h_4,
\end{align*}
as required.

Thus far we have verified that $\phi$ is a group homomorphism from $X_V$ to $W$. Let $H = \langle h_1,h_2,h_3,h_4\rangle$ and $G = \langle H, \rho, \sigma\rangle$ be the image of $\phi$. It is clear that $H=\C_{3^k}^4$. Note that $H$ is normal in $G$ (in fact, $H$ is normal in $W$) with the quotient $G/H$ isomorphic to $A_5$ (generated by $r = (1,2)(3,4)$ and $s = (1,3,5)$). Hence, $G = \C_{3^k}^4.\A_5$, and so $|X_U|\geq |G| = 3^{4k}\times 60$. As previously observed, $U$ is a quotient of $\C_{3^k}^4$ and in particular, $3^{4k}\geq |U|$. Finally, $U$ is normal in $X_V$ and $X_V/U\cong A_5$. In particular, $|X_V| = |U||A_5| \leq 3^{4k}\times 60$. It follows that $|X_U| = 3^{4k}\times 60 = |G|$ and consequently, $\phi$ is a group isomorphism.
\end{proof}

\begin{construction}\label{Cons3}
Fix an integer $k\geq 1$. Let $X$, $U$, $V$ and $N$ be as defined above, let $\,\overline{\phantom{\varphi}}\,$ be the quotient map from $X$ to $G\coloneqq X/N$, and let
\[
M=\langle u_1u_2,u_2u_3^{-1},u_3u_4,v_1v_2,v_2v_3^{-1}\rangle.
\]
Note that $M\cong \C_{3^k}^3\times\C_5^2$. Since it is straightforward to verify that $\langle\overline{x},\overline{y}[\overline{x},\overline{y}]^2\rangle=\C_2^2$ is a subgroup of $G$ stabilizing $M$, there exists a subgroup $Y=M{:}H=(\C_{3^k}^3\times\C_5^2){:}\C_2^2$ of $X$ with
\[
H\cong\overline{H}
=\langle\overline{x},\overline{y}[\overline{x},\overline{y}]^2\rangle.
\]
Let $W= \langle M,x\rangle = M{:}\C_2$ be a subgroup of $Y$.
\end{construction}

\begin{remark}
Although not needed in the sequel, it can be directly verified that one may take $H=\langle x,u_1^{-1}u_2^{-1}u_3u_4v_1^{-1}v_2^{-2}v_3^2y[x,y]^2\rangle$ in Construction~\ref{Cons3}.
\end{remark}

By Lemma~\ref{GenXk}, for a fixed integer $k\geq 1$, the group $X$ from Construction~\ref{Cons3} satisfies $X=N.G=(\C_{3^k}^4\times\C_5^3).\A_5$. We will see in the next lemma that $X$ is a non-split extension of $N=\C_{3^k}^4\times\C_5^3$ by $G=\A_5$.

\begin{lemma}\label{LemCon3}
For every integer $k\geq 1$, both $Y$ and $W$ are core-free in $X=(\C_{3^k}^4\times\C_5^3){}^{\boldsymbol{\cdot}}\A_5$, and $X\setminus N$ has no elements of order $3$ or $5$.
\end{lemma}

\begin{proof}
As in the proof of Lemma~\ref{lem-nop}, we see that $X\backslash N$ contains no elements of order $5$ (this can be also verified by \textsc{Magma}~\cite{BCP1997} via computation in $X_U$). As a consequence, $X=(\C_{3^k}^4\times\C_5^3){}^{\boldsymbol{\cdot}}\A_5$ is a non-split extension of $N$ by $\A_5$.

Let $J\cong \C_3^4$ be the subgroup of $N$ generated by the elements of order $3$ in $N$. To prove that $Y$ and $W$ are core-free in $X$, it suffices to show that the minimal normal subgroups of $X$ are precisely $J$ and $V$.

From the defining relations of $X$ in Construction~\ref{Cons3}, it is easy to see that the action of $X$ on $J$ and $V$ corresponds to the action of $\A_5$ on its fully deleted permutation modules over $\bbF_3$ and $\bbF_5$, respectively, which are irreducible. Therefore, both $J$ and $V$ are minimal normal subgroups of $X$. We next show that they are the only ones.

Let $K$ be a minimal normal subgroup of $X$. Since the quotient of $X$ by $N$ is a simple group, either $X = NK$ or $K\leq N$. If $X = NK$, then the minimality of $K$ implies $N \cap K = 1$ and so $X=N{:}K$, contradicting the conclusion that $X$ is a non-split extension of $N$ by $\A_5$. Thus, $K\leq N$. Note that a nontrivial element of $N$ is in $J$ if and only if it has order $3$, and is in $V$ if and only if it has order $5$. Therefore, $K$ must intersect $J$ or $V$ nontrivially. Hence $K=J$ or $V$, as desired.

It remains to prove that $X\backslash N$ contains no elements of order $3$. Consider $P \coloneqq \langle U, y\rangle$. Clearly, $U$ is normalized by $y$, and since $y^3\in U$ and $y\notin U$, we conclude that $|P| = 3|U| = 3^{4k+1}$. Hence, $P$ is a Sylow $3$-subgroup of $X$. To complete the proof, we need only show that $P\setminus U$ contains no elements of order $3$.
Observe from~\eqref{Eqn8} that $\langle u_1,u_3\rangle$ is normalized by $y$ and hence is normal in $P$, and that $P/\langle u_1,u_3\rangle$ is abelian. Thus, $P'\leq \langle u_1,u_3\rangle$.

Suppose for a contradiction that $uy^i\in P$ has order $3$ for some $u\in U$ and $i\in\{1,2\}$. Again by the Hall-Petrescu identity~\cite[Theorem~III.9.4]{Huppert2025}, we obtain
\[
u^3 (y^i)^3 = (uy^i)^3 c_2^3 c_3=c_2^3 c_3
\]
for some element $c_j$ in the $j$th term $P_j$ of the lower central series of $P$, where $j\in\{2,3\}$. Since $P_3 \leq P_2 = P'$, it follows that
\begin{equation}\label{EltOrder3}
u^3 (y^i)^3 \in P'\leq\langle u_1,u_3\rangle.
\end{equation}
Expressing $u = u_1^a u_2^b u_3^c u_4^d$ with $a,b,c,d\in \mathbb{Z}$ in~\eqref{EltOrder3} and recalling that $y^3 = u_2u_4^{-1}$, we  derive
\[
u_1^{3a}u_2^{3b+i}u_3^{3c}u_4^{3d-i}\in \langle u_1,u_3\rangle.
\]
This implies that $u_2^{3b+i}u_4^{3d-i}\in \langle u_1,u_3\rangle$, which is not possible as $i\in\{1,2\}$. The proof is thus complete.
\end{proof}

\begin{proof}[\rm\textbf{Proof of Theorem~\ref{ThmExample2}}]
Fix an integer $k\geq 1$. We follow the notation in Construction~\ref{Cons3} and identify $G$ with $\A_5$ via $\overline{x}\mapsto(1,2)(3,4)$ and $\overline{y}\mapsto(1,3,5)$. By Lemma~\ref{LemCon3}, we have a transitive permutation group $X$ with stabilizer $Y$ and degree $|X|/|Y|=3^{k+1}5^2$. To prove that $X$ is elusive, it suffices to show that each element of prime order in $N$ is conjugate in $X$ to some element of $M$.
Let $J=\langle w_1,w_2,w_3,w_4\rangle$, where $w_i\coloneqq u_i^{3^{k-1}}$ for $i\in\{1,2,3,4\}$. By Lemma~\ref{LemCon3}, every element of $X$ of order $3$ is in $J$. Similarly, every element of $X$ of order $5$ is in $V$. Hence, it suffices to show that each element of $J$ is mapped into $M\cap J= \langle w_1w_2,w_2w_3^{-1},w_3w_4\rangle=\C_3^3$  by some element of $G$ while each element of $V$ is mapped into $M\cap V=\langle v_1v_2,v_2v_3^{-1}\rangle=\C_5^2$ by some element of $G$. The latter holds for the same reason as \eqref{ItemHyp2}$\Rightarrow$\eqref{ItemHyp3} in the proof of Lemma~\ref{LemHyp} (it can be also verified by \textsc{Magma}~\cite{BCP1997} via computation in $X_U$). For the former, observe
\begin{equation}\label{Eqn11}
M\cap J=\langle w_1w_2,w_2w_3^{-1},w_3w_4\rangle=\langle w_1w_2,w_2w_4\rangle\times\langle w_1w_2^{-1}w_3^{-1}\rangle,
\end{equation}
and the action of $G=\A_5$ on $U$ satisfies
\begin{align}
(M\cap J)^{(1,4,5)}&=\langle w_1w_2u_4,w_2w_3^{-1},w_1u_3\rangle=\langle w_1w_2,w_4\rangle\times\langle w_1w_2^{-1}w_3^{-1}\rangle,\\
(M\cap J)^{(2,4,5)}&=\langle w_1w_2u_4,w_3^{-1}u_4,w_2w_3\rangle=\langle w_1,w_2w_4\rangle\times\langle w_1w_2^{-1}w_3^{-1}\rangle,\\
(M\cap J)^{(3,4,5)}&=\langle w_1w_2w_3,w_2w_4^{-1},w_3u_4\rangle=\langle w_1,w_2w_4^{-1}\rangle\times\langle w_1w_2^{-1}w_3^{-1}\rangle,\\
(M\cap J)^{(4,2,5)}&=\langle w_1w_4,w_3^{-1},w_2w_3w_4\rangle=\langle w_1w_4,w_2w_4\rangle\times\langle w_1w_2^{-1}w_3^{-1}\rangle,\\
\label{Eqn12}(M\cap J)^{(4,3,5)}&=\langle w_1w_2w_4,w_2,w_3w_4\rangle=\langle w_1w_4,w_2\rangle\times\langle w_1w_2^{-1}w_3^{-1}\rangle.
\end{align}
Denote the groups in~\eqref{Eqn11}--\eqref{Eqn12} by $W_1,\dots,W_6$ in order. The intersection of any two among $W_1,\dots,W_6$ has order $9$, and the intersection of any four among $W_1,\dots,W_6$ has order $3$. Moreover, for a $3$-subset $\{i,j,k\}$ of $\{1,\ldots,6\}$, we can directly verify that
\[
|W_i\cap W_j\cap W_k|=
\begin{cases}
9&\text{if }\{i,j,k\}=\{1,3,5\},\,\{1,4,6\},\,\{2,3,6\}\text{ or }\{2,4,5\}\\
3&\text{else}.
\end{cases}
\]
It follows from the Inclusion–Exclusion Principle that
\[
|W_1\cup\dots\cup W_6|=27\binom{6}{1}-9\binom{6}{2}+9\cdot4+3\left(\binom{6}{3}-4\right)-3\binom{6}{4}+3\binom{6}{5}-3\binom{6}{6}=81.
\]
Therefore, the images of $M\cap J$ under elements in $G$ cover $J$, or equivalently, each element of $J$ is mapped into $M\cap J$ by some element of $G$. This proves that the transitive permutation group $X=(\C_{3^k}^4\times\C_5^3){}^{\boldsymbol{\cdot}}\A_5$ with stabilizer $Y=(\C_{3^k}^3\times\C_5^2){:}\C_2^2$ is elusive.

Note that $\overline{W}$ contains an involution $\overline{x}$. Since there is only one conjugacy class of involutions in $G=\A_5$, the transitive permutation group $G$ with stabilizer $\overline{W}$ is $2$-elusive. Thereby we derive from Lemma~\ref{LemExt} that the transitive permutation group $X$ with stabilizer $W$ is $2$-elusive. Moreover, our proof shows that $X$ is both $3$-elusive and $5$-elusive. Hence the transitive permutation group $X=(\C_{3^k}^4\times\C_5^3){}^{\boldsymbol{\cdot}}\A_5$ with stabilizer $W=(\C_{3^k}^3\times\C_5^2){:}\C_2$ is elusive, which completes the proof.
\end{proof}

\begin{remark}
As previously explained, the presentation of the group $X$ was derived from the action of $\A_5$ on its fully deleted permutation modules over $\mathbb{F}_3$ and $\mathbb{F}_5$. In the proof of Theorem~\ref{ThmExample2}, we were able to use relations coming from more general modules for $\A_5$ over $\C_{3^k}$ with $k\geq 1$ an integer. However, the following shows that the same approach does not apply in the case of characteristic $5$.

Fix an integer $\ell \geq 1$. Let $X_\ell$ be a group with generators $u_1,u_2,u_3,u_4,v_1,v_2,v_3,x,y$ satisfying the relations in the presentation of the group $X$ from Construction \ref{Cons3} with the relation $v_i^5 = 1$ replaced by $v_i^{5^\ell}=1$.
Note that $V \coloneqq \langle v_1,v_2,v_3\rangle$ is an abelian normal subgroup of $X_\ell$ and a quotient of $\C_{5^\ell}^3$. Since $(xy)^5$ and $v_1$ are in $V$,
\[
v_1^{(xy)^5} = v_1.
\]
The left-hand side can be calculated directly as follows:
\begin{align*}
v_1^{xy} &= v_2^y = v_1^{-1}v_2, \\
v_1^{(xy)^2} &= (v_1^{-1}v_2)^{xy} = (v_2^{-1}v_1)^y = (v_1^{-1}v_2)^{-1}(v_1^{-1}v_3) = v_2^{-1}v_3, \\
v_1^{(xy)^3} &= (v_2^{-1}v_3)^{xy} = (v_1^{-2}v_2^{-1}v_3^{-1})^y = (v_1^{-1}v_3)^{-2} (v_1^{-1}v_2)^{-1}v_1 = v_1^4 v_2^{-1}v_3^{-2}, \\
v_1^{(xy)^4 }&= (v_1^4 v_2^{-1}v_3^{-2})^{xy} = (v_1v_2^6v_3^2)^y = (v_1^{-1}v_3)(v_1^{-1}v_2)^6(v_1^{-1})^{2} = v_1^{-9}v_2^6v_3, \\
v_1^{(xy)^5} &= (v_1^{-9}v_2^6v_3)^{xy} = (v_1^5v_2^{-10}v_3^{-1})^y = v_1^6 v_2^{-10}v_3^5.
\end{align*}
Hence we derive from $v_1^{(xy)^5} = v_1$ that
\begin{equation}\label{Rel1}
v_1^5v_2^{-10}v_3^5 = 1.
\end{equation}
Note that $(yx)^5 = y(xy)^5y^{-1} = ((xy)^5)^{y^{-1}}$. Since $(xy)^5\in V$ and $V$ is normal in $X_\ell$, we conclude that $(yx)^5\in V$, and so
\[
v_2^{(yx)^5} = v_2.
\]
The left-hand side can be calculated as follows by noting $(xy)^5 = x(yx)^5x$:
\[
v_2^{(yx)^5} = v_2^{x(yx)^5x} = \left(v_1^{(xy)^5}\right)^x = (v_1^6v_2^{-10}v_3^5)^x  = v_1^{-15}v_2v_3^{-5}.
\]
Thus, we derive from $v_2^{(yx)^5} = v_2$ that
\begin{equation}\label{Rel2}
    v_1^{15}v_3^{5} = 1.
\end{equation}
Moreover, $v_3$ and $(yx)^5$ also commute, implying that
\[
v_3^{(yx)^5} = v_3.
\]
Since $y^3\in U\coloneqq \langle u_1,u_2,u_3,u_4\rangle$, we deduce that $v^{y^3}=v$ and hence $v^{y^{-1}} = v^{y^2}$ for all $v\in V$. Therefore,
\[
    v_3^{(yx)^5} = v_3^{y(xy)^5y^{-1}} = \left(\left(v_1^{(xy)^5}\right)^{-1}\right)^{y^2} = (v_1^{-6}v_2^{10}v_3^{-5})^{y^2} = (v_1v_2^{10}v_3^{-6})^y = v_1^{-5}v_2^{10}v_3.
\]
We conclude that
\begin{equation}\label{Rel3}
    v_1^{-5}v_2^{10} = 1.
\end{equation}
By multiplying~\eqref{Rel1} and~\eqref{Rel3}, we obtain $v_3^5 =1$. This combined with~\eqref{Rel2} implies that $v_1^{15} = 1$, which is equivalent to $v_1^5 = 1$ (since $V$ is a quotient of a $5$-group). Plugging this back into~\eqref{Rel3} gives $v_2^{10} = 1$, further implying that $v_2^5 = 1$. Hence, $X_\ell$ is isomorphic to $X$ for every integer $\ell\geq 1$.
\end{remark}

\section{Concluding remarks}\label{SecRmk}

The following is a slight generalization of Wielandt~\cite[Lemma~8.4]{Wielandt1969} and can be proved by Wielandt's Dissection Theorem~\cite[Theorem~6.5]{Wielandt1969}.

\begin{theorem}[{\cite[Theorem~5.3]{CGJKKMN2002}}]\label{ThmE}
Let $E$ be a permutation group on a finite set $\Omega$ such that the following conditions hold:
\begin{enumerate}[{\rm(a)}]
\item\label{ItemE1} for each $v\in\Omega$, the stabilizer $E_v$ is normal in $E$;
\item\label{ItemE2} for each $u$ and $v$ in $\Omega$, the orders of $E_u$ and $E_v$ coincide, and either $E_u=E_v$ or $E=E_uE_v$.
\end{enumerate}
Then the $2$-closure of $E$ contains a derangement of prime order.
\end{theorem}

\begin{remark}
An equivalent condition to~\eqref{ItemE1} is that the induced permutation group of $E$ on each orbit is regular.
If condition~\eqref{ItemE1} holds, then $E_u=E_v$ for $u$ and $v$ in the same orbit of $E$.
Also, under condition~\eqref{ItemE1}, $|E_u|=|E_v|$ for each $u$ and $v$ in $\Omega$ if and only if $E$ is half-transitive (all orbits have the same length). In the conclusion of the theorem, the existence of a derangement of prime order is equivalent to the existence of a derangement that is a product of disjoint cycles of the same length.
\end{remark}

Let $X$ be a transitive permutation group with stabilizer $Y$. If $X$ has a normal elementary abelian subgroup $E=\C_p^\ell$ and $E\cap Y=\C_p^{\ell-1}$, then by Theorem~\ref{ThmE}, there is a derangement of prime order in the $2$-closure of $E$ and hence in the $2$-closure of $X$.
This shows that none of the elusive groups constructed in this paper is a counterexample to the Polycirculant Conjecture.

The elusive groups $N.G$ in Section~\ref{SecIdea} arise from $p'$-elusive groups $G$ that are nonabelian simple and act irreducibly on the elementary abelian $p$-group $N$ (regarded as an $\bbF_pG$-module). It is natural to look for elusive groups from other nonabelian simple $p'$-elusive groups $G$, or more generally from quasiprimitive (every nontrivial normal subgroup being transitive) $p'$-elusive permutation groups $G$.

\begin{problem}\label{ProbClassification}
Is there a counterexample to the Polycirculant Conjecture among elusive groups $N.G$ with stabilizer $Y$ such that $G$ is a quasiprimitive $p'$-elusive permutation group with stabilizer $YN/N$ for some prime $p$ and $N$ is an irreducible $\bbF_pG$-module?
\end{problem}

To approach Problem~\ref{ProbClassification}, we begin with the following observation.

\begin{lemma}\label{LemSplit}
Let $X=N{:}G$ be an elusive group with stabilizer $Y$ such that $G$ is a transitive permutation group with stabilizer $YN/N$. Then $G$ is elusive.
\end{lemma}

\begin{proof}
Denote by $\,\overline{\phantom{\varphi}}\,$ the quotient map from $X$ to $G=X/N$. Let $g\in G\leq X$ with $|g|=p$ for some prime $p$. Since $X$ is elusive, there exists $x\in X$ such that $g^x\in Y$. Accordingly, $\overline{x}$ is an element of $\overline{X}=G$ such that $\overline{g}^{\overline{x}}\in\overline{Y}$. Since $G$ is a transitive permutation group with stabilizer $YN/N=\overline{Y}$, it follows that $G$ is elusive.
\end{proof}

By~\cite{Giudici2003}, the only quasiprimitive elusive groups are the primitive wreath products $\M_{11}\wr K$, where $\M_{11}$ is a transitive subgroup of $\Sy_{12}$ and $K$ is any transitive permutation group, so the focus of Problem~\ref{ProbClassification} is on non-elusive groups $G$. In this case, Lemma~\ref{LemSplit} implies that the elusive group $N.G$ described in the problem must be a non-split extension of $N$ by $G$, and hence $\mathrm{H}^2(G,N)$ has positive dimension. This necessary condition imposes strong restrictions on the possibilities of $N$ in Problem~\ref{ProbClassification}, for any fixed $G$. Given a pair $(G,N)$, the cohomology machinery available in
\textsc{Magma}~\cite{BCP1997}, for example, can compute all extensions $N.G$, allowing one to determine whether any of them yield elusive groups.

Let $G$ be a quasiprimitive permutation group on a set $\Omega$. If $|\Omega|$ is a power of a prime $p$, then $G$ is $p'$-elusive. The quasiprimitive simple groups with prime power degree are classified in~\cite{Guralnick1983}. Thus, an interesting problem related to Problem~\ref{ProbClassification} is to determine whether such a group $G$ can be simple with degree equal to a power of $p$.
Another special case of $p'$-elusive groups $G$ is when $G$ contains a unique conjugacy class of derangements of prime order, all of which have order $p$. These are the so-called \emph{almost elusive}~\cite[Definition]{BH2022} quasiprimitive permutation groups, whose classification was recently completed in~\cite{Hall2024}. It would therefore be significant to first solve Problem~\ref{ProbClassification} for almost elusive groups $G$.

Finally, in view of the first known integers in $\mathcal{E}$ that are not divisible by $4$ as established by Theorem~\ref{ThmExample2}, we pose the following question.

\begin{problem}\label{QuesSmallDegree}
What is the smallest odd integer in $\mathcal{E}$ and what is the smallest integer in $\mathcal{E}$ that is congruent to $2$ modulo $4$?
\end{problem}

\section*{Acknowledgments}

We thank the referee for careful and insightful review of our manuscript.
Chen was supported by the National Natural Science Foundation of China (12101518) and by the Fundamental Research Funds for the Central Universities (20720210036, 20720240136). Lee was supported by an Australian Research Council Discovery Early Career Researcher Award (project number DE230100579). Part of the work was done during Mitrović's visit to Lee at Monash University and Xia at the University of Melbourne, supported by the Marsden Fund of New Zealand Grant 21-UOA-174 and University of Auckland's DRDF Grant.  O’Brien was supported by the Marsden Fund of New Zealand Grant 23-UOA-080 and by a Research Award of the Alexander von Humboldt Foundation. Xia was supported by ARC Discovery Project DP250104965.


\begin{thebibliography}{}

\bibitem{AAS2019}
M.~Arezoomand, A.~Abdollahi and P.~Spiga,
On problems concerning fixed-point-free permutations and on the polycirculant conjecture---a survey,
\textit{Trans.~Comb.}, 8 (2019), no.~1, 15--40.

\bibitem{AG2021}
M.~Arezoomand and M.~Ghasemi,
On $2$-closed elusive permutation groups of degrees $p^2q$ and $p^2qr$,
\textit{Comm.~Algebra}, 49 (2021), no.~2, 614--620.

\bibitem{BM1968}
H.~Behr and J.~Mennicke,
A presentation of the groups $\mathrm{PSL}(2,p)$,
\textit{Canadian J.~Math.}, 20 (1968), 1432--1438.

\bibitem{BCP1997}
W. Bosma, J. Cannon and C. Playoust,
The Magma algebra system I: The user language,
\textit{J.~Symbolic Comput.}, 24 (1997), no.~3-4, 235--265.

\bibitem{BH2022}
T.~C.~Burness and E.~V.~Hall,
Almost elusive permutation groups,
\textit{J.~Algebra}, 594 (2022), 519--543.

\bibitem{Burnside1897}
W.~Burnside,
\textit{Theory of groups of finite order}, Cambridge University press, London, 1897.

\bibitem{Cameron1997}
P.~J.~Cameron (ed.),
Research problems---Problems from the Fifteenth British Combinatorial Conference,
\textit{Discrete Math.}, 167/168 (1997), 605--615.

\bibitem{CGJKKMN2002}
P.~J.~Cameron, M.~Giudici, G.~Jones, W.~M.~Kantor, M.~Klin, D.~Maru\v{s}i\v{c} and L.~A.~Nowitz,
Transitive permutation groups without semiregular subgroups,
\textit{J.~London Math.~Soc.~(2)}, 66 (2002), 325--333.

\bibitem{CCNPW1985}
J.~H.~Conway, R.~T.~Curtis, S.~P.~Norton, R.~A.~Parker and R.~A.~Wilson,
\textit{Atlas of finite groups: maximal subgroups and ordinary characters for simple groups}, Clarendon Press, Oxford, 1985.

\bibitem{DS2005}
F.~Diamond and J.~Shurman,
\textit{A first course in modular forms}, Grad.~Texts in Math., 228, Springer-Verlag, New York, 2005.

\bibitem{DM2011}
E.~Dobson and D.~Maru\v{s}i\v{c},
On semiregular elements of solvable groups,
\textit{Comm.~Algebra}, 39 (2011) 1413--1426.

\bibitem{FKS1981}
B.~Fein, W.~M.~Kantor and M.~Schacher,
Relative Brauer groups II,
\textit{J.~Reine Agnew.~Math.}, 328 (1981), 39--57.

\bibitem{Giudici2003}
M.~Giudici,
Quasiprimitive groups with no fixed point free elements of prime order,
\textit{J.~London Math.~Soc.~(2)}, 67 (2003), 73--84.

\bibitem{Giudici2007}
M.~Giudici,
New constructions of groups without semiregular subgroups,
\textit{Comm.~Algebra}, 35 (2007), no.~9, 2719--2730.

\bibitem{GK2009}
M.~Giudici and S.~Kelly,
Characterizing a family of elusive permutation groups,
\textit{J.~Group Theory}, 12 (2009) 95--105.

\bibitem{GMPV2015}
M.~Giudici, L.~Morgan, P.~Poto\v{c}nik, G.~Verret,
Elusive groups of automorphisms of digraphs of small valency,
\textit{European J.~Combin.}, 46 (2015), 1--9.

\bibitem{Gorenstein1980}
D.~Gorenstein,
\textit{Finite groups}, Second edition, Chelsea Publishing Co., New York, 1980.

\bibitem{Guralnick1983}
R.~M.~Guralnick,
Subgroups of prime power index in a simple group,
\textit{J.~Algebra}, 81 (1983), no.~2, 304--311.

\bibitem{Hall2024}
E.~V.~Hall,
The quasiprimitive almost elusive groups,
\textit{Internat.~J.~Algebra Comput.}, 34 (2024), no.~5, 739--777.

\bibitem{Huppert2025}
B.~Huppert,
\textit{Finite Groups I}, Translated by C.~A.~Schroeder, Grundlehren der mathematischen Wissenschaften, 364, Springer Cham, 2025.

\bibitem{Jordan1872}
C.~Jordan,
Recherches sur les substitutions,
\textit{J.~Math.~Pures Appl.~(Liouville)}, 17 (1872), 351--367.

\bibitem{Jordan1988}
D.~Jordan,
Eine Symmetrieeigenschaft von Graphen,
in: \textit{Graphentheorie und ihre Anwendungen (Stadt Wehlen, 1988)}, 17--20 (in German).

\bibitem{KL1990}
P. B. Kleidman and M. W. Liebeck,
\textit{The subgroup structure of the finite classical groups}, Cambridge University Press, Cambridge, 1990.

\bibitem{KP1993I}
A.~Kleshchev and A.~Premet,
On second degree cohomology of symmetric and alternating groups,
\textit{Comm.~Algebra}, 21 (1993), no.~2, 583--600.

\bibitem{KP1993II}
A.~Kleshchev and A.~Premet,
Corrigenda: ``On second degree cohomology of symmetric and alternating groups'',
\textit{Comm.~Algebra}, 21 (1993), no.~9, 3401.

\bibitem{Klin1998}
M.~Klin,
On transitive permutation groups without semi-regular subgroups,
\textit{ICM 1998: International Congress of Mathematicians, Berlin, 18--27 August 1998}, Abstracts of short communications and poster sessions, 1998, 279.

\bibitem{Marusic1981}
D.~Maru\v{s}i\v{c},
On vertex symmetric digraphs,
\textit{Discrete Math.}, 36 (1981) 69--81.

\bibitem{MS1998}
D.~Maru\v{s}i\v{c} and R.~Scapellato,
Permutation groups, vertex-transitive digraphs and semiregular automorphisms,
\textit{European J.~Combin.}, 19 (1998), no.~6, 707--712.

\bibitem{McQuillan1965}
D.~L.~McQuillan,
Classification of normal congruence subgroups of the modular group.
\textit{Amer.~J.~Math.}, 87 (1965), 285--296.

\bibitem{Serre2003}
J.-P.~Serre,
On a theorem of Jordan,
\textit{Bull.~Amer.~Math.~Soc.}, 40 (2003), 429--440.

\bibitem{UoGVAG2012}
University of Georgia VIGRE Algebra Group,
Second cohomology for finite groups of Lie type,
\textit{J.~Algebra}, 360 (2012), 21--52.

\bibitem{Wielandt1969}
H.~Wielandt,
Permutation groups through invariant relations and invariant functions,
\textit{Lecture Notes}, Ohio State University, 1969.

\bibitem{Xu2009}
J.~Xu,
On elusive permutation groups of square-free degree,
\textit{Comm.~Algebra}, 37 (2009), no.~9, 3200--3206.

\bibliographystyle{100}
\end{thebibliography}
\end{document}